\documentclass{amsart}
\pdfoutput=1
\usepackage[utf8]{inputenc}
\usepackage{amsmath, amsthm, amssymb, graphicx}
\usepackage{dsfont}
\usepackage[justification=centering]{caption}
\usepackage{float}
\usepackage{enumitem}
\usepackage{listings}
\usepackage{soul}
\allowdisplaybreaks
\usepackage[colorlinks=true, citecolor=magenta, linkcolor=blue]{hyperref}
\usepackage[letterpaper,left=1.2in,right=1.2in,bottom=1.3in]{geometry}
\usepackage[backend=biber,style=numeric,sorting=nty]{biblatex}
\usepackage{xcolor}
\usepackage{setspace}
\usepackage{paracol}
\DeclareMathOperator\sign{sgn}

\theoremstyle{definition}
\newtheorem{thm}{Theorem}[section]

\newtheorem{lem}{Lemma}[section]
\newtheorem{rmk}{Remark}[section]

\newtheorem{conj}{Conjecture}[section]

\definecolor{codegreen}{rgb}{0,0.6,0}
\definecolor{codegray}{rgb}{0.5,0.5,0.5}
\definecolor{codepurple}{rgb}{0.58,0,0.82}
\definecolor{backcolour}{rgb}{0.95,0.95,0.92}
 
\lstdefinestyle{mystyle}{
    backgroundcolor=\color{backcolour},
    commentstyle=\color{codegreen},
    keywordstyle=\color{magenta},
    numberstyle=\tiny\color{codegray},
    stringstyle=\color{codepurple},
    basicstyle=\ttfamily\footnotesize,
    breakatwhitespace=false,         
    breaklines=true,                 
    captionpos=b,                    
    keepspaces=true,                 
    numbers=left,                    
    numbersep=5pt,                  
    showspaces=false,                
    showstringspaces=false,
    showtabs=false,                  
    tabsize=2
}
 
\title[A Study of the Carathéodory Conjecture]{A Study of the Carathéodory Conjecture through Non-Rotationally Symmetric Surfaces}

\begin{document}

\author[Lucy Cai]{Jiaying Cai}
\address{Phillips Exeter Academy, Exeter, NH, 03833}
\email{jcai1@exeter.edu}

\begin{abstract}
Carathéodory's well-known conjecture states that every sufficiently smooth, closed convex surface in three dimensional Euclidean space admits at least two umbilic points. It has been established that the conjecture is true for all rotationally symmetric surfaces; in this paper, we investigate the umbilic points of two families of surfaces without rotational symmetry, and compute their indices. In particular, we find that the family of surfaces of the form $ax^{2k}+by^{2k}+cz^{2k}=1$ with $a,b,c>0$, $k\in\mathds{Z}_{>1}$ admit 14 umbilic points: six of one known form and eight of another. For many tested values of $a,b,c,k$, such umbilic points have indices $-1/2$ and $1$, respectively. We also explore the dependence of the umbilic points on the parameter $\epsilon$ of the surface $ax^2+\epsilon x^4+ay^2+\epsilon y^4+bz^2=1$. In particular, for both $a<b$ and $a>b,$ there exist exactly two umbilic points with index 1 for $\epsilon$ smaller than certain critical values. For larger $\epsilon,$ surfaces with $a>b$ admit exactly ten umbilic points; for many tested values of $a,b,\epsilon,$ these points have indices 1/2 and -1. For larger $\epsilon,$ surfaces with $a<b$ admit eighteen umbilic points; for many tested values of $a,b,\epsilon,$ these points have indices -1/2 and 1.\\

\end{abstract}

\maketitle
\tableofcontents
\onehalfspacing
\section{Introduction}
Constantin Carathéodory posted his famous conjecture in a 1924 session of the Berlin mathematical society. Since then, it has subsequently appeared in many works and problem lists, notably that of S.T. Yau \cite{MR0645762} (page 684, problem 64). Before we state the conjecture, we briefly recall several concepts.

Consider a point \textit{p} on a closed convex smooth surface \textit{$M$} in $\mathds{R}^3$ (in this paper, we will use ``smooth" to denote $C^{\infty}$). Every plane containing the unit normal vector at \textit{p} defines a \textit{normal curvature}, the maximum and minimum of which are denoted the \textit{principal curvatures} $k_1, k_2$. The directions at which $k_1, k_2$ point are denoted the \textit{principal directions}, and are always orthogonal. At an umbilic point \textit{x}, the principal curvatures are equal, i.e., $k_1=k_2$. In other words, \textit{M} is locally spherical at \textit{x}, and it follows that \textit{x} is a singularity of the principal direction field. The gaussian and mean curvatures are defined as $k_1k_2$ and $\frac{k_1+k_2}{2}$, respectively. In this paper, we use $E,F,G$ to denote the coefficients of the first fundamental form \textit{I} and $e,f,g$ to denote the coefficients of the second fundamental form \textit{II}. The principal curvatures and directions at each point \textit{p} are given by the shape operator \textit{S}, defined in terms of the fundamental form coefficients: 
$$S = (EG-F^2)^{-1}\begin{pmatrix}
eG-fF&fG-gF\\
fE-eF&gE-fF\\
\end{pmatrix}.$$
The index of an umbilic point \textit{x} is defined to be the index of its principal direction field about \textit{x}. See \cite{MR3676571} (15.1, page 153) for a formal definition.

The Poincaré-Hopf index theorem states that the sum of the indices of all  umbilic points on a surface equals its Euler characteristic. All closed convex sufficiently smooth surfaces in $\mathds{R}^3$ have Euler characteristic two; thus, Carathéodory’s conjecture is closely tied to Loewner’s conjecture, which states that every isolated umbilic point has index less than or equal to one. A proof of Loewner’s conjecture, together with the Poincaré-Hopf Theorem, implies the truth of Carathéodory’s conjecture; most attempts at proving Carathéodory’s conjecture take this route. Bol \cite{MR12489} and Hamburger \cite{MR1052, MR0006480, MR0006481} were the first to prove the conjecture for the real analytic case, although doubts were later expressed and the results were reexamined by Klotz \cite{MR0120602}. Ghomi and Howard have written a paper in which they use Mobius inversions to create closed convex smooth and umbilic free surfaces in the complement of one point (and get arbitrarily close to a sphere). Further historical results can be found in \cite{MR2957223}.

The Carathéodory conjecture is nearly one century old, and has resisted numerous attacks even for the real analytic case. To settle the conjecture, one really needs to understand some nontrivial examples. The first type of surfaces are perhaps the simplest non-rotationally symmetric smooth convex surfaces. The second type of surfaces are small perturbations of ellipsoids. These two types of surfaces are all nontrivial examples. It is therefore very natural for us to study them.

It is known that all rotationally symmetric surfaces have at least two umbilic points. See, for example, Hilbert and Cohn-Vossen \cite{MR0046650} (page 203). In this paper, we aim to shed light on the conjecture by exploring the umbilic points of several non-rotationally symmetric surfaces. Umehara and Yamada \cite{MR3676571} (page 163, Example 15.8) discuss the example of the non-rotationally symmetric ellipsoid $ax^2+by^2+cz^2=1$ with $a,b,c$ distinct. This example has four umbilic points with index 1/2. 

In our paper, we generalize this example to surfaces of the form $ax^{2k}+by^{2k}+cz^{2k}=1$. In Section \ref{sec:2}, we compute the number and location of the umbilic points of such surfaces, as well as their indices for several tested values of $a,b,c,$ and $k$. In Section \ref{sec:3}, we explore umbilic points of the surface $ax^2+\epsilon x^4 + ay^2+\epsilon y^4+bz^2=1$ and study how their number and location shift at critical values of $\epsilon$.
\section{A Simple Non-Rotationally Symmetric Surface}
\label{sec:2}
\noindent In this section, we explore the family of surfaces of the form $ax^{2k} + by^{2k} + cz^{2k} = 1$, where $a, b, c>0, k\in\mathds{Z}_{>1}$. We compute the number of umbilic points of such surfaces, as well as their locations and indices.
\begin{thm}
All surfaces of the form $ax^{2k} + by^{2k} + cz^{2k} = 1$, where $a, b, c>0, k\in\mathds{Z}_{>1}$ admit fourteen umbilic points: six of the form $\big\{(\pm a^{-\frac{1}{2k}}, 0, 0), (0, \pm b^{-\frac{1}{2k}}, 0), (0, 0, \pm c^{-\frac{1}{2k}})\big\}$ and eight of the form $\Big\{\left(\pm \left(\frac{bc}{a}\right)^{\frac{1}{2k}}\left(bc+ca+ab\right)^{-\frac{1}{2k}}, \pm\left(\frac{ac}{b}\right)^{\frac{1}{2k}}\left(bc+ca+ab\right)^{-\frac{1}{2k}}, \pm\left(\frac{ab}{c}\right)^{\frac{1}{2k}}\left(bc+ca+ab\right)^{-\frac{1}{2k}}\right)\Big\}.$
\end{thm}
\noindent For reference, we include two instances of the surface below. Visible umbilic points are highlighted in red.
\begin{figure}[H]
\begin{minipage}{.45\textwidth}
    \centering
    \includegraphics[scale=0.5]{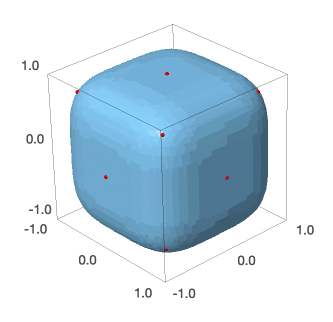}
    \caption{$x^4+y^4+z^4=1$}
\end{minipage}
\begin{minipage}{.45\textwidth}
    \centering
    \includegraphics[scale=0.5]{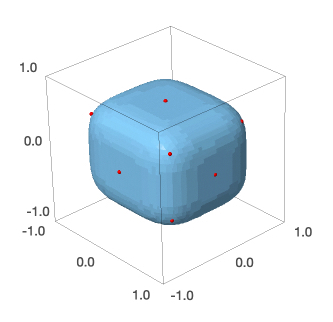}
    \caption{$2x^4+3y^4+5z^4=1$}
\end{minipage}
\end{figure}
\begin{proof}
We split the proof into two steps. In the first, we prove that the surface satisfies the given requirements. In the second, we prove the number and location of the umbilic points. In an additional third step, we illustrate the indices of the umbilic points for several instances of the surface and conjecture as to its general form.

\noindent \textbf{Step 1.} First, we prove all such surfaces are closed, convex, and smooth. To do so, we first prove all such surfaces are regular. It is possible to express all such surfaces as $f^{-1}(0)$, where $$f(x, y, z) = ax^{2k} + by^{2k} + cz^{2k} -1.$$ Note that $f$ is infinitely differentiable. Also, $0$ is a regular value of $f$, since its partial derivatives $f_x = 2akx^{2k-1}$, $f_y = 2bky^{2k-1}$, $f_z = 2ckz^{2k-1}$ only vanish simultaneously at $(0, 0, 0)$, which is not contained in $f^{-1}(0)$. It follows from the inverse function theorem \cite{MR3837152} that all such surfaces are regular and smooth. 

In order to check for convexity, we first compute the first and second fundamental coefficients. We parameterize the surface as $\mathcal{S} = (u, v,(1-au^{2k}-bv^{2k})^{\frac{1}{2k}}),$ which covers the top half strictly above $z=0$. Refer to appendix \hyperlink{appendix:1}{[1]} for detailed calculations of the fundamental form coefficients, which exist as long as $1-au^{2k}-bv^{2k}\neq 0$ and $a^2u^{4k-2} + b^2v^{4k-2} + c^{\frac{1}{k}}(-au^{2k} - bv^{2k} + 1)^{2-\frac{1}{k}} \neq 0$. The first equation is the boundary of our parameterization: we do not need to worry about this since it is possible to switch $a$, $b$, $c$ to check for umbilic points along this boundary. The second equation has no solutions (since the left side is the sum of square roots that cannot simultaneously be 0).

Now, to prove all such surfaces are convex, we introduce a lemma.
\begin{lem}[Convexity]
    A closed surface in $\mathds{R}^3$ is convex if and only if its Gaussian curvature is nonnegative everywhere.
\end{lem}
\begin{proof}
    The proof of this lemma is due to the Chern-Lashof Theorem; see \cite{MR3837152} (p. 387, Remark 3). 
\end{proof}
\noindent We compute the Gaussian curvature $K,$ which is given by
$$K = \frac{\det(II)}{\det(I)}= \frac{eg-f^2}{EG-F^2}.$$
By symmetry, we only need to consider the open parameterization, so that $au^{2k}+bv^{2k}<1.$ Note that since $\sqrt{c^{\frac{1}{k}}(1-au^{2k}-bv^{2k})^{2-\frac{1}{k}} + a^2u^{4k-2} + b^2v^{4k-2}}>0$,
\begin{align*}
    \sign(eg-f^2)&=\sign(abu^{2k-2}v^{2k-2}(2k-1)^2(1-au^{2k})(1-bv^{2k})-a^2b^2(2k-1)^2u^{4k-2}v^{4k-2}) \\
    &=\sign(abu^{2k-2}v^{2k-2}(2k-1)^2((1-au^{2k})(1-bv^{2k})-abu^{2k}v^{2k})) \\
    &=\sign(abu^{2k-2}v^{2k-2}(2k-1)^2(1-au^{2k}-bv^{2k})) \\
    &\geq0.
\end{align*}
Similarly, since $c^{\frac{1}{k}}(-au^{2k}-bv^{2k}+1)^{2-\frac{1}{k}}\geq0$,
\begin{align*}
    \sign(EG-F^2)
    &= \sign(c^{\frac{1}{k}}(-au^{2k}-bv^{2k}+1)^{2-\frac{1}{k}}(a^2u^{4k-2}+b^2v^{4k-2}+c^{\frac{1}{k}}(-au^{2k}-bv^{2k}+1)^{2-\frac{1}{k}}))=1.
\end{align*}
So, $\sign(K)\geq0$ and all such surfaces are convex.

\noindent \textbf{Step 2.} We now compute the number of umbilic points of all such surfaces. To do so, we first introduce a lemma. \begin{lem}[Weingarten matrix]
The umbilic points of a surface occur precisely where the Weingarten matrix $$ C = A^{-1}B = \begin{pmatrix}
E&F\\
F&G\\
\end{pmatrix}^{-1}\begin{pmatrix}
e&f\\
f&g\\
\end{pmatrix}$$ is a scalar multiple of the identity matrix.
\end{lem}
\begin{proof}
This is well-known; see \cite{MR3676571} (Proposition 9.6, page 94). 
\end{proof} 
\noindent We then solve the following equations manually:
\begin{align}
    &C[0][1] = C[1][0] = 0 \\
    &C[0][0] = C[1][1]
\end{align}
To verify that $C$ exists, we check that $EG-F^2\neq 0$. This quantity is equal to 0 only when
\begin{align*}
    &(a^2u^{4k-2}+c^{\frac{1}{k}}(-au^{2k}-bv^{2k}+1)^{2-\frac{1}{k}})(b^2v^{4k-2}+c^{\frac{1}{k}}(-au^{2k}-bv^{2k}+1)^{2-\frac{1}{k}})-(abu^{2k-1}v^{2k-1})^2=0 \Rightarrow \\
    &c^{\frac{1}{k}}(-au^{2k}-bv^{2k}+1)^{2-\frac{1}{k}}(a^2u^{4k-2}+b^2v^{4k-2}+c^{\frac{1}{k}}(-au^{2k}-bv^{2k}+1)^{2-\frac{1}{k}})=0 \Rightarrow \\ 
    &1-au^{2k}-bv^{2k}=0 \hspace{0.5cm} \textrm{or} \hspace{0.5cm} a^2u^{4k-2}+b^2v^{4k-2}+c^{\frac{1}{k}}(-au^{2k}-bv^{2k}+1)^{2-\frac{1}{k}}=0.
\end{align*}
Again, we do not need to worry about these two equations.
Now, we expand and simplify $C$ to get
$$C[0][0] = \frac{Ge-Ff}{EG-F^2}, \hspace{0.4cm} C[0][1] = \frac{Gf-Fg}{EG-F^2}, \hspace{0.4cm} C[1][0] = \frac{Ef-eF}{EG-F^2}, \hspace{0.4cm} C[1][1] = \frac{Eg-Ff}{EG-F^2}.$$
Equation (1) becomes $Gf=Fg$ and $Ef=eF$. We rearrange and simplify to get
\begin{align*}
    &Gf=Fg \Rightarrow abu^{2k-1}v^{2k-1}(b^2v^{4k-2}+c^{\frac{1}{k}}(-au^{2k}-bv^{2k}+1)^{2-\frac{1}{k}} - bv^{2k-2}(-au^{2k}+1)) = 0\Rightarrow \\ 
    &abu^{2k-1}v^{2k-1}(c^{\frac{1}{k}}(-au^{2k}-bv^{2k}+1)^{2-\frac{1}{k}} - bv^{2k-2}(-au^{2k}-bv^{2k}+1))=0\Rightarrow \\
    &u^{2k-1}v^{2k-1}(-au^{2k}-bv^{2k}+1)(c^{\frac{1}{k}}(-au^{2k}-bv^{2k}+1)^{1-\frac{1}{k}}-bv^{2k-2})=0.
\end{align*}
Similarly, $Ef=eF$ becomes
$$u^{2k-1}v^{2k-1}(-au^{2k}-bv^{2k}+1)(c^{\frac{1}{k}}(-au^{2k}-bv^{2k}+1)^{1-\frac{1}{k}}-au^{2k-2})=0.$$
Equation (2) becomes $Ge=Eg$. We get
\begin{align*}
    &(b^2v^{4k-2}+c^{\frac{1}{k}}(-au^{2k}-bv^{2k}+1)^{2-\frac{1}{k}})au^{2k-2}(-bv^{2k}+1)\\&=(a^2u^{4k-2}+c^{\frac{1}{k}}(-au^{2k}-bv^{2k}+1)^{2-\frac{1}{k}})bv^{2k-2}(-au^{2k}+1).
\end{align*}
Simultaneously solving these equations gives the solutions
\begin{align*}
    \left(u,v\right) = \hspace{0.2cm} \left(0,0\right), \hspace{0.2cm} \tfrac{(bc)^{\frac{1}{k-1}}}{(ab)^{\frac{1}{k-1}}+(bc)^{\frac{1}{k-1}}+(ca)^{\frac{1}{k-1}}}\left(\pm a^{-\frac{1}{2k}},\pm b^{-\frac{1}{2k}}\right).
\end{align*}
Taking advantage of symmetry gives the following general form for the umbilic points:
\begin{align*}
    (x, y, z) = \> &\Big(\pm a^{-\frac{1}{2k}}, 0, 0\Big), \hspace{0.2cm} \Big(0, \pm b^{-\frac{1}{2k}}, 0\Big), \hspace{0.2cm} \Big(0, 0, \pm c^{-\frac{1}{2k}}\Big),  \\
    &\tfrac{(bc)^{\frac{1}{k-1}}}{(ab)^{\frac{1}{k-1}}+(bc)^{\frac{1}{k-1}}+(ca)^{\frac{1}{k-1}}}\left(\pm a^{-\frac{1}{2k}},\pm b^{-\frac{1}{2k}},\pm c^{-\frac{1}{2k}}\right).
\end{align*}
\end{proof}

\textbf{Index of Umbilics.}
Now, to find the indices of the umbilic points, we study the shape of the principal direction field near singularities.
\begin{lem}[Lines of curvature] \label{lem:1.3}
The equation for lines of curvature can be written as $$\begin{vmatrix}
(v')^2&-u'v'&(u')^2\\
E&F&G\\
e&f&g\\
\end{vmatrix}=0.$$
Expanding gives an alternate form:
$$(fE-eF)(u')^2+(gE-eG)u'v'+(gF-fG)(v')^2=0.$$
\end{lem} 
\begin{proof}
See \cite{MR3837152} (section 3-3, page 161). 
\end{proof}
We compute the quantities $fE-eF$, $gE-eG$, $gF-fG$. Dividing each by $2k-1$ gives
$$fE-eF = (a^2u^{4k-2}+c^{\frac{1}{k}}(-au^{2k}-bv^{2k}+1)^{2-\frac{1}{k}})abu^{2k-1}v^{2k-1} -  a^2bu^{4k-3}v^{2k-1}(-bv^{2k}+1),$$
\begin{multline*}
    gE-eG = bv^{2k-2}(-au^{2k}+1)(a^2u^{4k-2}+c^{\frac{1}{k}}(-au^{2k}-bv^{2k}+1)^{2-\frac{1}{k}})\\-au^{2k-2}(-bv^{2k}+1)(b^2v^{4k-2}+c^{\frac{1}{k}}(-au^{2k}-bv^{2k}+1)^{2-\frac{1}{k}}),
\end{multline*}
$$gF-fG = (-au^{2k}+1)(ab^2u^{2k-1}v^{4k-3})-abu^{2k-1}v^{2k-1}(b^2v^{4k-2}+c^{\frac{1}{k}}(-au^{2k}-bv^{2k}+1)^{2-\frac{1}{k}}).$$
We now divide both sides by $(v')^2$. Viewing $u$ as a function of $v$, we plug the differential equation into the computer algebra software Mathematica using its numerical differential equation solving method \href{https://reference.wolfram.com/language/ref/NDSolve.html}{NDSolve}. In the following example, we set $a=1,b=1,c=1,k=2.$ Refer to appendix \hyperlink{appendix:2}{[2]} for relevant code and see Figures \ref{output:1}, \ref{output:2}, \ref{output:3} for resulting graphs given initial conditions.
\begin{figure}[H]
\begin{minipage}{.32\textwidth}
    \centering
    \includegraphics[scale=0.35]{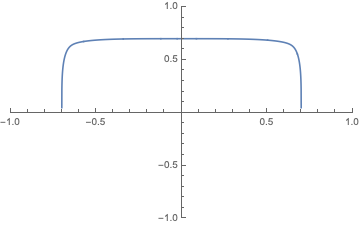}
    \caption{\\$(v,u)=(0,0.7)$ \\ $\{v,-0.7,0.7\}$}
    \label{output:1}
\end{minipage}
\begin{minipage}{.34\textwidth}
    \centering
    \includegraphics[scale=0.35]{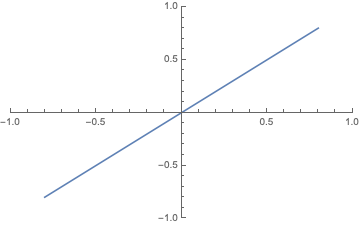}
    \caption{\\$(v,u)=(0.8,0.8)$ \\ $\{v,-0.8,0.8\}$}
    \label{output:2}
\end{minipage}
\begin{minipage}{.32\textwidth}
    \centering
    \includegraphics[scale=0.35]{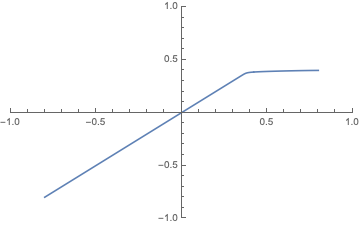}
    \caption{\\$(v,u)=(0.8,0.4)$ \\ $\{v,-0.8,0.8\}$}
    \label{output:3}
\end{minipage}
\end{figure}
\noindent Overlaying these graphs with other solution curves gives an image of the lines of curvature near $(0,0,1)$ on the surface $x^4+y^4+z^4=1$. We can test a variety of surfaces by changing the input values for $a,b,c,k$: refer to appendix \hyperlink{appendix:3}{[3]} for results. In particular, we see that setting $a=b=1, k=2$ and testing values of $c$ up to 100 gives the same approximate principal field lines. 
In all of the demonstrated cases, the shape of the lines of curvature remain the same: thus, the indices of the umbilic points remain the same, and are independent of $a,b,c,$ or $k$. To this end, we offer the following conjecture.

\begin{conj} \label{conj:1.1}
The index of the umbilic points on the surface $ax^{2k}+by^{2k}+cz^{2k}=1$ for $a,b,c>0, k\in\mathds{Z}_{>1}$ are independent of $a,b,c,$ and $k$.
\end{conj}

\noindent At the umbilic points, the vector fields look approximately as follows:
\begin{figure}[H]
\begin{minipage}{.32\textwidth}
    \centering
    \includegraphics[scale=0.2]{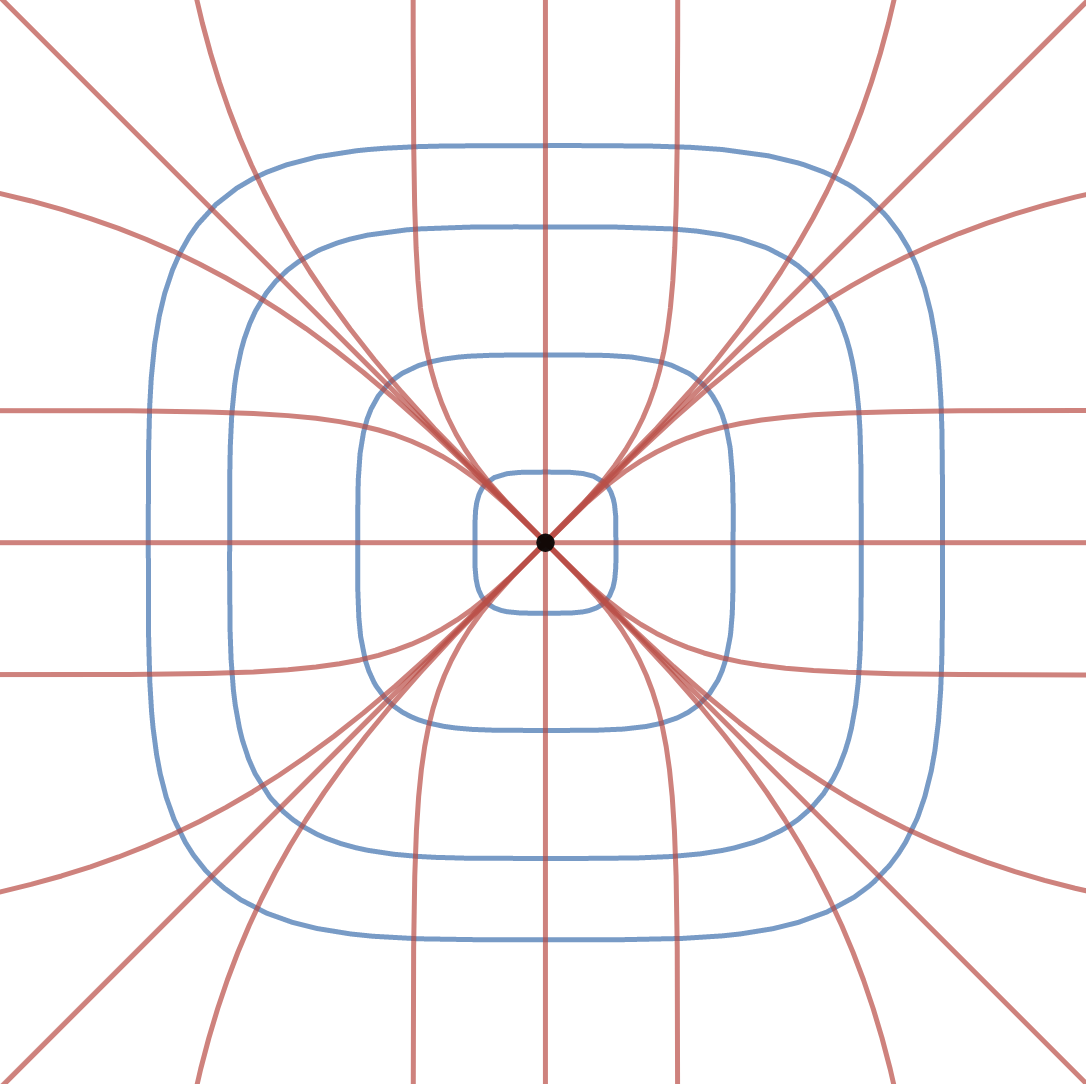}
    \caption{Vector field with index 1.}
    \label{figure:vfield6}
\end{minipage}
\begin{minipage}{.32\textwidth}
    \centering
    \includegraphics[scale=0.18]{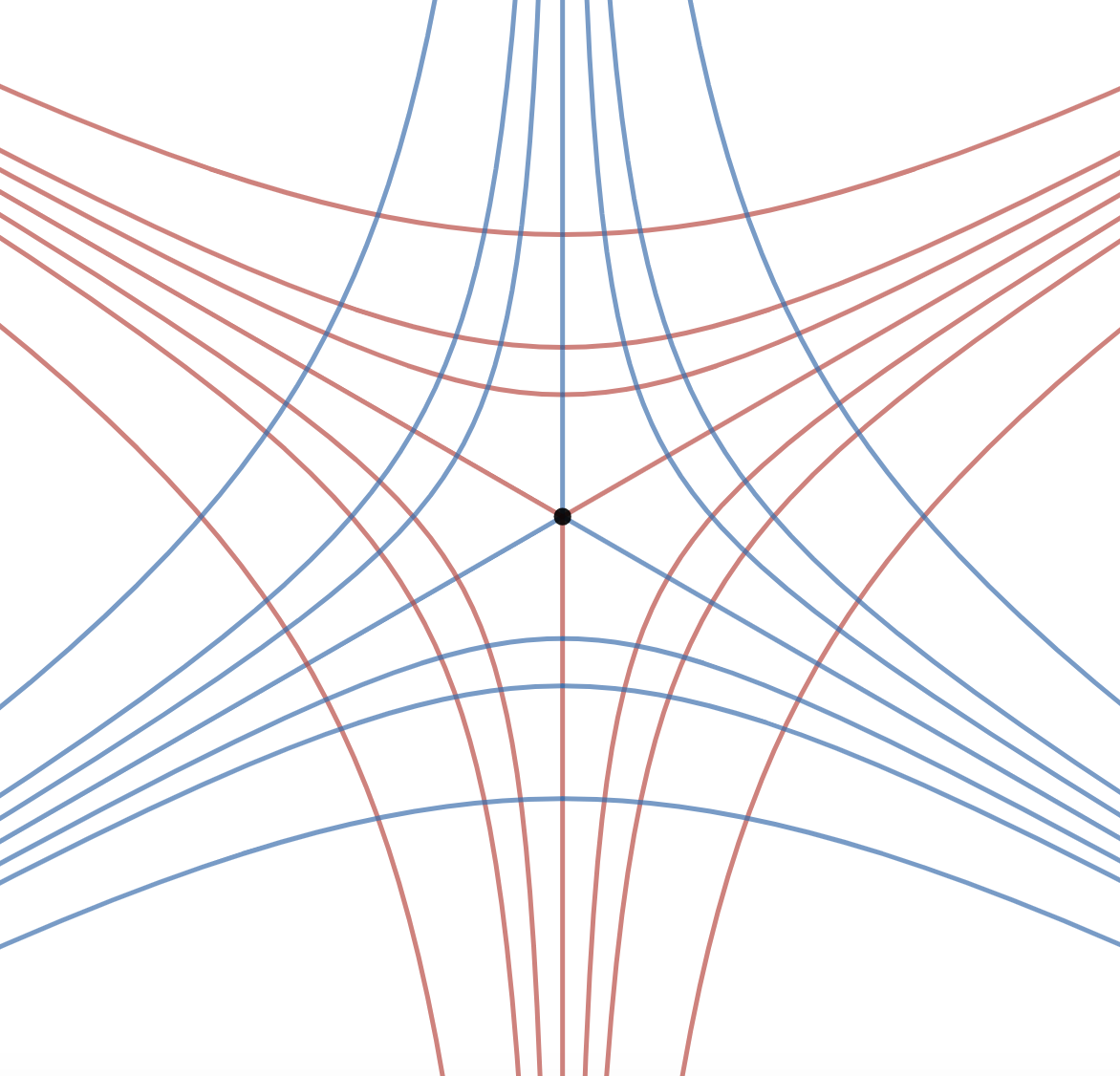}
    \caption{Vector field with index -$\frac{1}{2}$.}
    \label{figure:vfield7}
\end{minipage}
\begin{minipage}{.32\textwidth}
    \centering
    \includegraphics[scale=0.45]{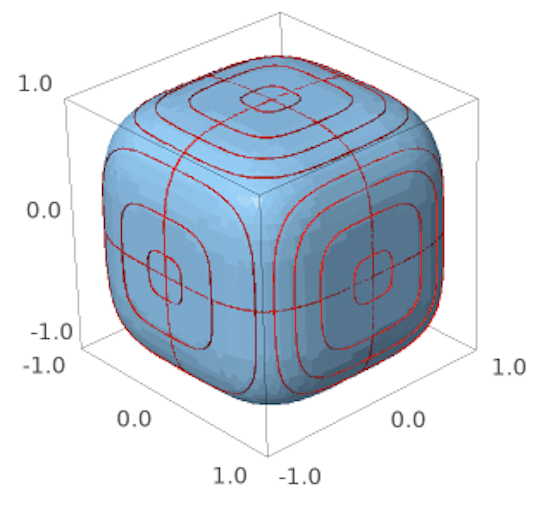}
    \caption{Lines of curvature.}
    \label{figure:vfield8}
\end{minipage}
\end{figure}
\noindent We can deduce by inspection that these vector fields have index 1 and $-\frac{1}{2}$, respectively. For verification, we compare Figure \ref{figure:vfield} with \cite{MR3676571} (section 15, page 156) and check that these values agree with the Poincaré-Hopf index theorem.
\begin{rmk}
This theorem readily extends to surfaces of the form $ax^{2k}+by^{2k}+cz^{2k}=R$ for $R>0$ through a scaling of the constants $a,b,c.$
\end{rmk}
\begin{rmk}
Note that surfaces of the form $ax^{2k+1} + by^{2k+1} + cz^{2k+1} = 1$ are not closed, since any of $x,y,z$ can extend infinitely in the negative direction. Thus, such surfaces are not of interest, and it is safe to claim we have investigated the umbilic points of all surfaces of the form $ax^k+by^k+cz^k=1$ with $a,b,c>0,k\in\mathds{Z}^+$. 
\end{rmk}
\begin{rmk}
We can verify the results of NDSolve by plotting the logs of the residuals. Note that solving with higher precision (red) improves upon error, as compared to machine precision (black). See Appendix \hyperlink{appendix:4}{[4]} for the code.
\begin{figure}[H]
\begin{minipage}{.3\textwidth}
    \centering
    \includegraphics[scale=0.35]{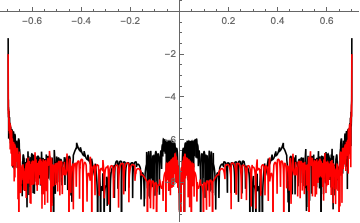}
    \caption*{$(v,u)=(0,0.7)$ \\ $\{v,-0.7,0.7\}$}
    \label{residual:1}
\end{minipage}
\begin{minipage}{.3\textwidth}
    \centering
    \includegraphics[scale=0.35]{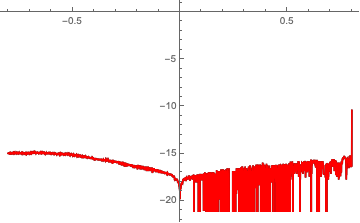}
    \caption*{$(v,u)=(0.8,0.8)$ \\ $\{v,-0.8,0.8\}$}
    \label{residual:2}
\end{minipage}
\begin{minipage}{.3\textwidth}
    \centering
    \includegraphics[scale=0.35]{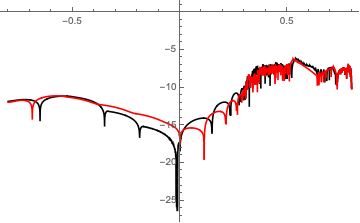}
    \caption*{$(v,u)=(0.8,0.4)$ \\ $\{v,-0.8,0.8\}$}
    \label{residual:3}
\end{minipage}
\end{figure}
\noindent The above graphs correspond to the outputs \ref{output:1}, \ref{output:2}, \ref{output:3} respectively. We see that the error mostly remains smaller than a $10^{-5}$ magnitude.
\end{rmk}

\section{A Perturbation of the Ellipsoid}
\label{sec:3}
In this section, we explore the family of surfaces of the form $ax^2 + \epsilon x^4 + ay^2 + \epsilon y^4 + bz^2 = 1$, where $a, b, \epsilon>0$. We investigate the dependence of the number and location of umbilic points on $\epsilon$. Note that when $\epsilon=0$, this surface reduces to an ellipsoid of revolution  $ax^2+ay^2+bz^2=1$. Refer to Figures \ref{fig:surface2.1}, \ref{fig:surface2.2}, \ref{fig:surface2.3}, \ref{fig:surface2.4} for several instances.
\begin{thm}
Consider the surface $ax^2 + \epsilon x^4 + ay^2 + \epsilon y^4 + bz^2 = 1$ where $a,b,\epsilon\geq0.$ Then,   
\[
  \{\text{\# of umbilic points} \mid a>b\} =
  \begin{cases}
                                2 & \text{if $\epsilon\leq\frac{a^2}{6}\left(\frac{a}{b}-1\right)$} \\
                                10 & \text{if $\epsilon>\frac{a^2}{6}\left(\frac{a}{b}-1\right).$}
  \end{cases}\]
  When $\epsilon\leq\frac{a^2}{b}\left(\frac{a}{b}-1\right),$ these points occur at $(0,0,\pm \sqrt{1/b})$ and have index 1. When $\epsilon>\frac{a^2}{b}\left(\frac{a}{b}-1\right),$ these points occur at $(0,0,\pm \sqrt{1/b})$ and \label{location of a>b} $$\left(0,\pm\left(\frac{-a+\sqrt{3b(2a+b)^{-1}(a^2+4e)}}{2\epsilon}\right)^\frac{1}{2},\pm\left(\frac{(a-b)(a^2+4\epsilon)}{2b\epsilon (2a+b)}\right)^{\frac{1}{2}}\right),$$ $$\left(\pm\left(\frac{-a+\sqrt{3b(2a+b)^{-1}(a^2+4e)}}{2\epsilon}\right)^\frac{1}{2},0,\pm\left(\frac{(a-b)(a^2+4\epsilon)}{2b\epsilon (2a+b)}\right)^{\frac{1}{2}}\right).$$ Also,
\[
  \{\text{\# of umbilic points} \mid a<b\} =
  \begin{cases}
                                2 & \text{if $\epsilon<\frac{(5a+b)(b-a)}{18}$} \\
                                18 & \text{if $\epsilon>\frac{(5a+b)(b-a)}{18}$}
  \end{cases}\]
  When $\epsilon<\frac{(5a+b)(b-a)}{18},$ these points occur at $(0,0,\pm \sqrt{1/b})$ and have index 1. When $\epsilon>\frac{(5a+b)(b-a)}{18},$ these points occur at $(0,0,\pm \sqrt{1/b})$ and $$\left(\pm\sqrt{\frac{b-a}{6\epsilon}},\pm\sqrt{\frac{b-a}{6\epsilon}},\pm\left(\frac{5a^2-4ab-b^2+18\epsilon}{18b\epsilon}\right)^{\frac{1}{2}}\right).$$
\end{thm}
\begin{proof}
Again, we split the proof into two steps. In the first step, we show that this surface is closed, convex, and sufficiently smooth. In the second step, we compute the umbilic points and their indices.

\noindent \textbf{Step 1.} We first show the surface is regular. This surface is expressible as $f^{-1}(0)$, where
$$f(x,y,z)=ax^2 + \epsilon x^4 + ay^2 + \epsilon y^4 + bz^2-1.$$
$f$ is infinitely differentiable. Its partial derivatives $f_x=2ax+4\epsilon x^3$, $f_y=2ay+4\epsilon y^3$, $f_z=2bz$ vanish simultaneously only at $(x,y,z)=(0,0,0)$, which is not contained in $f^{-1}(0)$. Thus, 0 is a regular value of $f$, and this surface is regular. To prove this surface is convex, we compute its fundamental coefficients. We parameterize the top half as $\mathcal{S}=(u,v,b^{-\frac{1}{2}}(1-au^2-\epsilon u^4-av^2-\epsilon v^4)^\frac{1}{2}).$ Refer to appendix \hyperlink{appendix:5}{[5]} for detailed calculations of the fundamental form coefficients, which are defined as long as $1-au^2-\epsilon u^4-av^2-\epsilon v^4\neq 0$ and $1+\frac{1}{4}b^{-1}(1-au^2-\epsilon u^4-av^2-\epsilon v^4)^{-1}\left((2au+4\epsilon u^3)^2+(2av+4\epsilon v^3)^2\right)\neq 0,$ which is always true. 

Now, to prove this surface is convex, we compute the sign of its Gaussian curvature $K$: $$K=\frac{eg-f^2}{EG-F^2}.$$
Some examination gives $\sign(eg-f^2)> \sign((au+2\epsilon u^3)^2(av+2\epsilon v^3)^2-(au+2\epsilon u^3)^2(av+2\epsilon v^3)^2)=0.$ 
For the same reason, $\sign(EG-F^2)>0,$ so $\sign(K)=1$ and all such surfaces are convex.

\noindent \textbf{Step 2.} We use the same method as in the first surface. $Ef=eF$ becomes $$uv(a+2\epsilon u^2)(a+2\epsilon v^2)\Big(-1+\tfrac{1}{b}(a+6\epsilon u^2)\Big)=0.$$ $Gf=gF$ becomes $$uv(a+2\epsilon u^2)(a+2\epsilon v^2)\Big(-1+\tfrac{1}{b}(a+6\epsilon v^2)\Big)=0.$$
We now split this problem up into two cases. When $a>b,$ we have two possible solutions: $u=0$ or $v=0.$ When $a<b,$ we have three possible solutions: $u=0$ or $v=0$ or $u=\pm v=\pm \sqrt{\frac{b-a}{6\epsilon}}.$ 

\noindent First, consider $a>b.$ Plugging in $u=0$ to $Ge=gE$ gives the following polynomial in $v^2$: $$v^6\Big(-2\epsilon^2-\frac{4a}{b}\epsilon^2\Big)+v^4\Big(-2a\epsilon -\frac{4a^2}{b}\epsilon\Big)+v^2\left(a^2-\frac{a^3}{b}+6\epsilon\right)=0.$$ Clearly, $v=0$ is a solution. Factoring out $v^2$ and solving gives $$v^2=-\frac{a}{2\epsilon}\pm\frac{1}{2\epsilon}\left(\frac{b}{2a+b}\right)\sqrt{3\left(\frac{2a+b}{b}\right)(a^2+4\epsilon)},$$ which give real nonzero solutions for $v$ as long as \begin{align*}&-\frac{a}{2\epsilon}+\frac{1}{2\epsilon}\left(\frac{b}{2a+b}\right)\sqrt{3\left(\frac{2a+b}{b}\right)(a^2+4\epsilon)}>0 \\ &\Longrightarrow 3(a^2+4\epsilon)< a^2\left(1+\frac{2a}{b}\right) \\ &\Longrightarrow \epsilon > \frac{a^2}{6}\left(\frac{a}{b}-1\right).\end{align*}
For $\epsilon > \frac{a^2}{6}\left(\frac{a}{b}-1\right),$ we have two nonzero solutions for $v,$ indicating two umbilic points in the case $u=0.$ Their locations given in Theorem \ref{location of a>b} can be computed by substituting in $v$ and solving for the $z$ coordinate. By symmetry, the case $v=0$ gives two umbilic points as well, for a total of ten umbilic points. 

For $\epsilon \leq \frac{a^2}{6}\left(\frac{a}{b}-1\right),$ we have no nonzero solutions for $v,$ and by symmetry, no nonzero solutions for $u$ when $v=0$ as well. This indicates that the only umbilic points are the two that occur when $u,v=0.$ See Figures \ref{fig:surface2.1} and \ref{fig:surface2.2} for umbilic points in black on $\frac{1}{2}x^2+\epsilon x^4+\frac{1}{2}y^2+\epsilon y^4+\frac{1}{5}z^2=1.$
\begin{figure}[H]
\begin{minipage}{.49\textwidth}
    \centering
    \includegraphics[scale=0.43]{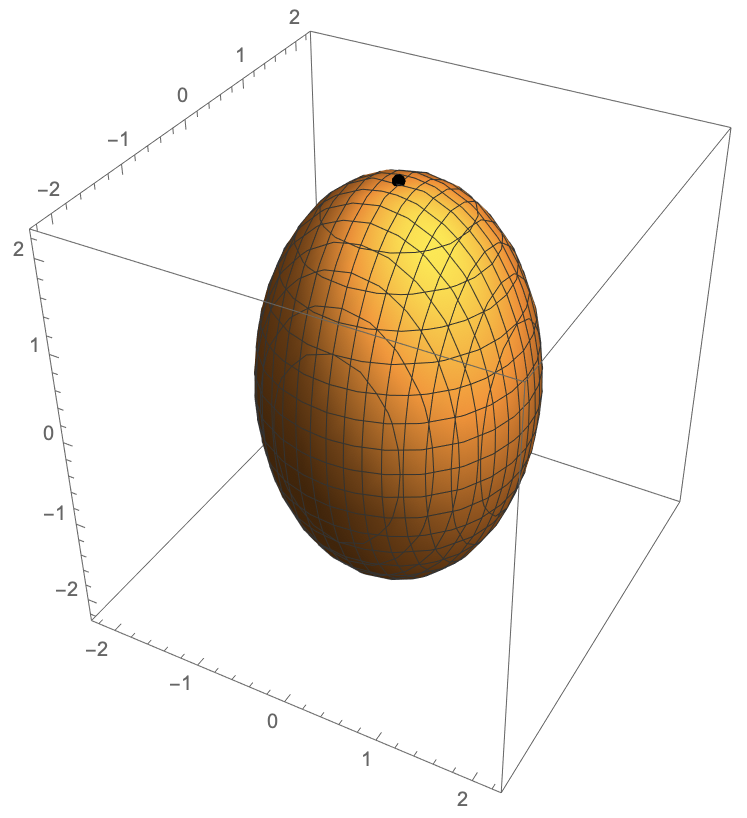}
    \caption{Umbilic points for small $\epsilon$}
    \label{fig:surface2.1}
\end{minipage}
\begin{minipage}{.49\textwidth}
    \centering
    \includegraphics[scale=0.44]{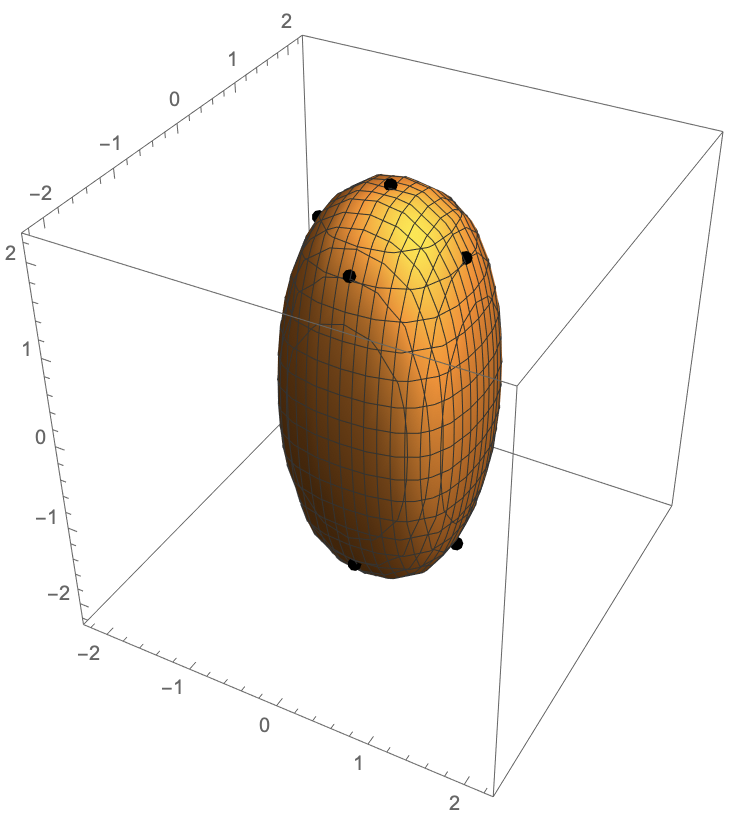}
    \caption{Umbilic points for larger $\epsilon$}
    \label{fig:surface2.2}
\end{minipage}
\end{figure}
Now, consider $a<b.$ Plugging in $u=0$ to $Ge=gE$ gives the same polynomial as above, where $v=0$ is again a solution. Factoring out $v^2$ and solving gives $$v^2=-\frac{a}{2\epsilon}+\frac{1}{2\epsilon}\left(\frac{b}{2a+b}\right)\sqrt{3\left(\frac{2a+b}{b}\right)(a^2+4\epsilon)}.$$ However, we can prove that these solutions always lie outside of the surface. When $u=0,$ the largest $v$ value occurs when $z=0:$ $\epsilon y^4+ay^2-1=0$ gives $$v^2=\frac{-a+\sqrt{a^2+4\epsilon}}{2\epsilon}.$$
We get the following:
\begin{align*}
    a<b &\Longrightarrow \frac{2a}{b}<2 \Longrightarrow \sqrt{1+\frac{2a}{b}} < \sqrt{3} \Longrightarrow 1+\frac{2a}{b}<\sqrt{3\left(1+\frac{2a}{b}\right)} \\ &\Longrightarrow \frac{-a+\sqrt{a^2+4\epsilon}}{2\epsilon}<-\frac{a}{2\epsilon}+\frac{1}{2\epsilon}\left(\frac{b}{2a+b}\right)\sqrt{3\left(1+\frac{2a}{b}\right)(a^2+4\epsilon)}.
\end{align*}
Thus, the only umbilic point for the case $u=0$ is given by $u,v=0,$ which by symmetry is also true of the case $v=0.$ Finally, consider the case $u=\pm v=\pm \sqrt{\frac{b-a}{2\epsilon}}.$ By symmetry in $u$ and $v$, $Ge=gE$ is always satisfied so long as $u=v=\sqrt{\frac{b-a}{2\epsilon}}$ actually lies on the surface. Along the line $u=v,$ the furthest point on the surface is given by $2\epsilon x^4+2ax^2-1=0,$ or $$x=\left(\frac{-a+\sqrt{a^2+2\epsilon}}{2\epsilon}\right)^{\frac{1}{2}}.$$ So, these solutions exist as long as
\begin{align*}
    &\sqrt{\frac{b-a}{6\epsilon}}<\left(\frac{-a+\sqrt{a^2+2\epsilon}}{2\epsilon}\right)^{\frac{1}{2}} \\ &\Longrightarrow \frac{b-a}{3}< -a+\sqrt{a^2+2\epsilon} \\ &\Longrightarrow 2a+b< 3\sqrt{a^2+2\epsilon} \\ &\Longrightarrow 4a^2+4ab+b^2< 9(a^2+2\epsilon) \\ &\Longrightarrow \epsilon > \frac{(5a+b)(b-a)}{18}.
\end{align*}
Thus, for $\epsilon > (5a+b)(b-a)/18,$ there are two umbilics that occur when $u,v=0$ and eight umbilics that occur when $u=\pm v=\pm \sqrt{\frac{b-a}{2\epsilon}}.$ For $\epsilon < (5a+b)(b-a)/18,$ there are exactly two umbilics that occur when $u,v=0.$ See Figures \ref{fig:surface2.3} and \ref{fig:surface2.4} for umbilic points in black on $\frac{1}{5}x^2+\epsilon x^4+\frac{1}{5}y^2+\epsilon y^4+\frac{1}{2}z^2=1.$
\begin{figure}[H]
\begin{minipage}{.48\textwidth}
    \centering
    \includegraphics[scale=0.43]{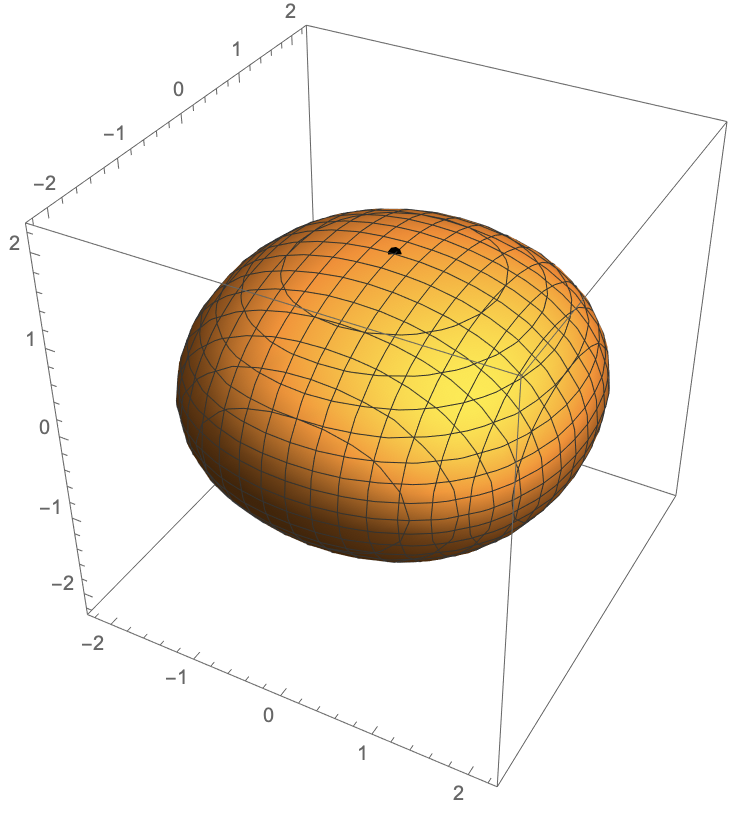}
    \caption{Umbilic points for small $\epsilon$}
    \label{fig:surface2.3}
\end{minipage}
\begin{minipage}{.48\textwidth}
    \centering
    \includegraphics[scale=0.44]{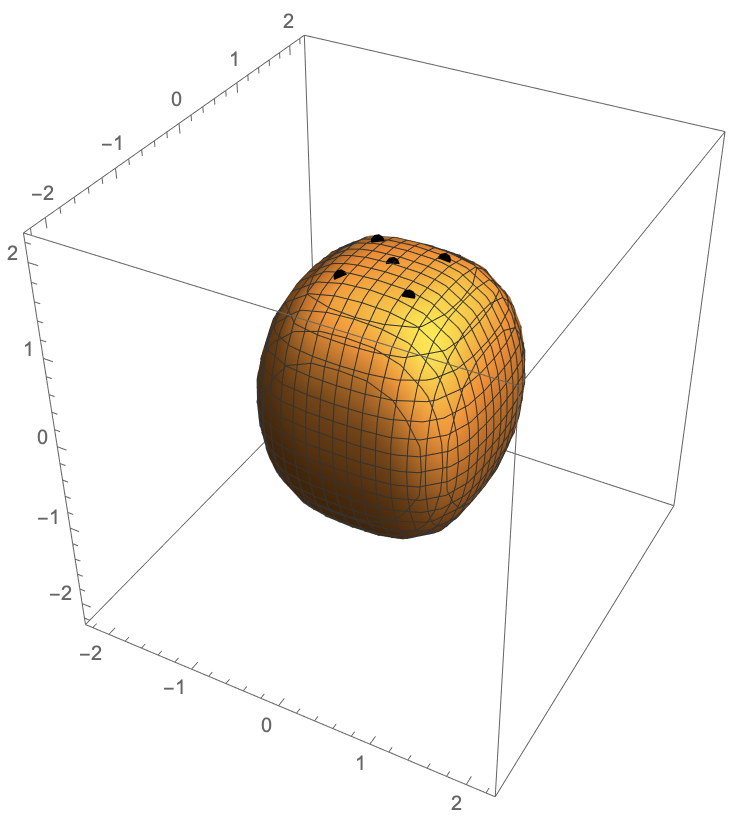}
    \caption{Umbilic points for larger $\epsilon$}
    \label{fig:surface2.4}
\end{minipage}
\end{figure}
\noindent We can verify all of the above results by checking computationally.

Now, we have found all umbilic points on our top and bottom half parametrizations, but still are missing the equator of the surface ($z=0$). To check this equator for umbilics, we consider a new rotated surface $ax^2+\epsilon x^4+by^2+az^2+\epsilon z^4=1$ parameterized by $\mathcal{S}=(u,v,(2\epsilon)^{-\frac{1}{2}}(-a+(a^2+4\epsilon(1-au^2-\epsilon u^4-av^2-\epsilon v^4))^\frac{1}{2})).$ Since we only need to check the boundary of our previous parameterization for umbilic points, we can substitute $y=0$ in all computations. Refer to appendix \hyperlink{appendix:7}{[7]} for detailed computations of the fundamental form coefficients. Note that after substitution, $F=f=0$, so that $Gf=Fg$ and $Ef=Fe$ are always satisfied. $Eg=Ge$ then becomes $$u^2(a+2\epsilon u^2)^2(6\epsilon Q+a-b)+(2\epsilon Q+a)^2Q(a-b+6\epsilon u^2)=0,$$ where
$$Q=\frac{\sqrt{a^2+4\epsilon(1-au^2-\epsilon u^4)}-a}{2\epsilon}.$$
For our parameterization, $Q>0.$ Thus, for the surface $a>b,$ the quantity on the left is nonnegative, indicating no solutions on the equator. We have now proved the following: \[
  \{\text{\# of umbilic points} \mid a>b\} =
  \begin{cases}
                                2 & \text{if $\epsilon<\frac{a^2}{b}\left(\frac{a}{b}-1\right)$} \\
                                10 & \text{if $\epsilon>\frac{a^2}{b}\left(\frac{a}{b}-1\right).$}
  \end{cases}\] In the first case, the two umbilics must have index 1 by Poincaré-Hopf. In the second case, we can graph the lines of curvature as we did with the first surface. Doing so for $a=0.516,b=0.3,\epsilon=0.1$ gives the below graphs. Refer to appendix \hyperlink{appendix:8}{[8]} for relevant code.  
\begin{figure}[H]
\begin{minipage}{.45\textwidth}
    \centering
    \includegraphics[scale=0.4]{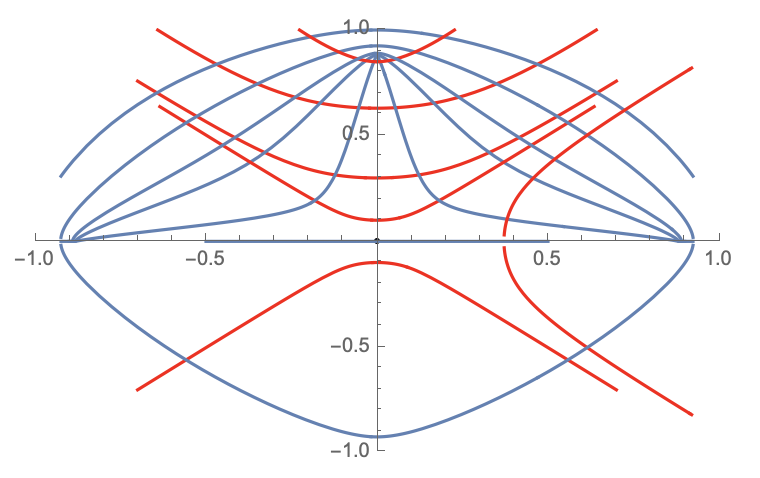}
    \caption{Lines of curvature around $u,v=0$ umbilic}
    \label{figure:vfield13}
\end{minipage}
\begin{minipage}{.45\textwidth}
    \centering
    \includegraphics[scale=0.4]{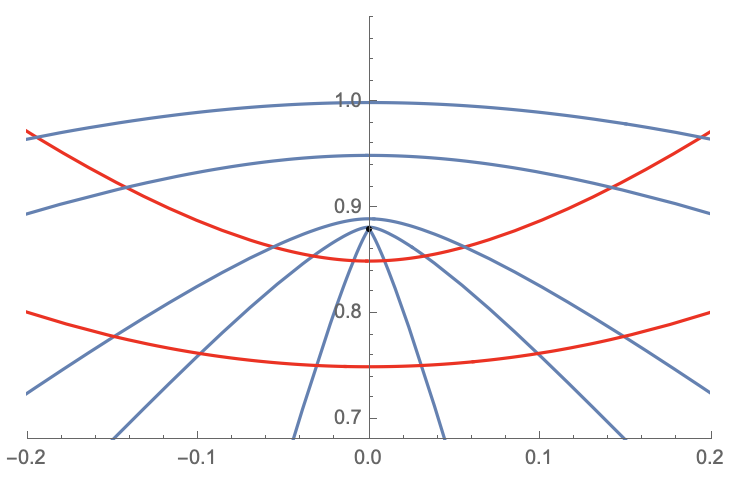}
    \caption{Lines of curvature around $u=0$ umbilic}
    \label{figure:vfield14}
\end{minipage}
\end{figure}
\noindent From these, it is clear that there are then two umbilics with index $-1$ and eight with index 1/2.
\begin{conj}\label{conj1}
Independent of $a$ and $b,$ the surface $ax^2+\epsilon x^4+ay^2+\epsilon y^4+bz^2=0$ with $a>b$ has exactly two umbilics with index $-1$ and eight with index 1/2 when $\epsilon > \frac{a^2}{b}\left(\frac{a}{b}-1\right).$
\end{conj}

For the surface $a<b,$ we can solve this equation in Mathematica. For $\epsilon<(5a+b)(b-a)/18,$ the below code gives no real solutions for $u,$ indicating no umbilic points on the equator: \begin{lstlisting}[language=Mathematica,columns=fullflexible]
(*Define Q for sake of simplicity*)
Q = (Sqrt[a^2 + 4*e*(1 - a*u^2 - e*u^4)] - a)/(2*e)
(*Solve equation symbolically assuming conditions*)
Assuming[{b > a > 0, (5*a + b)*(b - a)/18 > e > 0, 1 - a*u^2 - e*u^4 > 0}, Simplify[Solve[u^2*(a + 2*e*u^2)^2*(6*e*Q + a - b) + (2*e*Q + a)^2*Q*(a - b + 6*e*u^2) == 0, u, Reals]]]
\end{lstlisting}
Similarly, for $\epsilon>(5a+b)(b-a)/18,$ the below code gives exactly four solutions for $u,$ indicating eight total umbilic points on the equator:
\begin{lstlisting}[language=Mathematica,columns=fullflexible,firstnumber=3]
Assuming[{b > a > 0, (5*a + b)*(b - a)/18 < e, 1 - a*u^2 - e*u^4 > 0},Simplify[Solve[{u^2*(a + 2*e*u^2)^2*(6*e*Q + a - b) + (2*e*Q + a)^2*Q*(a - b + 6*e*u^2) == 0, u < Sqrt[-(a/e) + Sqrt[a^2 + 4 e]/e]/Sqrt[2]}, u, Reals]]]\end{lstlisting}
We have now proved the following: \[
  \{\text{\# of umbilic points} \mid a<b\} =
  \begin{cases}
                                2 & \text{if $\epsilon<\frac{(5a+b)(b-a)}{18}$} \\
                                18 & \text{if $\epsilon>\frac{(5a+b)(b-a)}{18}$}
  \end{cases}\]
In the first case, the two umbilics must have index 1 by Poincaré-Hopf. In the second case, we can graph the lines of curvature again. Doing so for $a=0.3, b=0.516, \epsilon=0.1$ gives the below graphs.
\begin{figure}[H]
\begin{minipage}{.49\textwidth}
    \centering
    \includegraphics[scale=0.4]{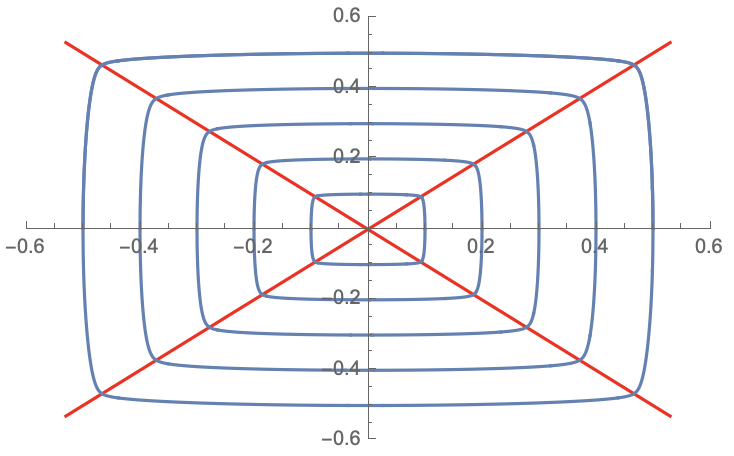}
    \caption{Lines of curvature around $u,v=0$ umbilic}
    \label{figure:vecfield}
\end{minipage}
\begin{minipage}{.49\textwidth}
    \centering
    \includegraphics[scale=0.4]{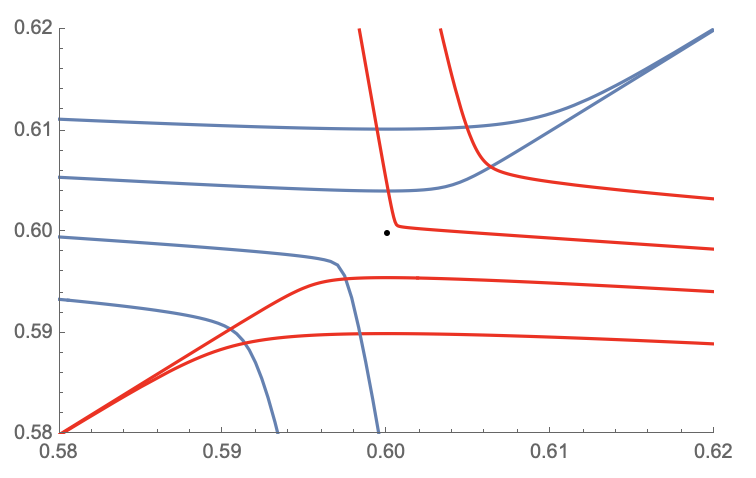}
    \caption{Lines of curvature around $u=v$ umbilic}
    \label{figure:vfield}
\end{minipage}
\end{figure}
From these, it is clear that there are then two umbilics with index 1 and eight with index $-1/2.$ The eight equator umbilics must then have index 1/2.
\begin{conj}\label{conj2}
Independent of $a$ and $b,$ the surface $ax^2+\epsilon x^4+ay^2+\epsilon y^4+bz^2=0$ with $a<b$ has exactly two umbilics with index $1,$  eight with index $-1/2,$ and eight with index $1/2$ when $\epsilon > (5a+b)(b-a)/18.$
\end{conj}
This concludes the proof.
\end{proof}
\begin{rmk}
    In a previous version of this paper, we studied the example $\frac{x^2}{4} + \epsilon x^4 + \frac{y^2}{4} + \epsilon y^4 + \frac{z^2}{9} = 1$ for small $\epsilon>0.$ The exploration of the family of surfaces $ax^2 + \epsilon x^4 + ay^2 + \epsilon y^4 + bz^2 = 1$ as a generalization of that one was mentioned as an open question in the previous paper, and in the semifinal presentations of the 2019 Yau High School Science Awards. We solved the open question this year.
\end{rmk}
\begin{rmk}
    The surface $ax^2 + \epsilon x^4 + ay^2 + \epsilon y^4 + bz^2 = 1$ is only closed and convex for $a,b,\epsilon>0.$ It is always symmetric across the $x,y$ axes and $z=0$ plane. However, for $b<0,$ the surface does not cross the $z=0$ plane, so it must be disconnected. For $\epsilon<0,$ the $\epsilon x^4$ and $\epsilon y^4$ can be infinitely negative, so all such surfaces must also be open. The surface with $a<0$ and $b,\epsilon >0$ is concave. This is easy to see by setting $x$ to be 0 and implicitly taking the second $y$-derivative of $bz^2= -ay^2 - \epsilon y^4 + 1,$ which gives $\partial^2z/\partial y^2=-a/bz$.  Thus, it is safe to assume $a,b,\epsilon>0$ is the only meaningful surface to study.
\end{rmk}


\noindent Similarly, we would like to prove or disprove Conjectures \ref{conj1} and \ref{conj2}:
\begin{enumerate}[noitemsep,topsep=0pt]
    \item Does the surface $ax^2+\epsilon x^4+ay^2+\epsilon y^4+bz^2=0$ with $a>b$ and $\epsilon > \frac{a^2}{b}\left(\frac{a}{b}-1\right)$ have exactly two umbilics with index $-1$ and eight with index 1/2, independent of $a$ and $b$?
    \item Does the surface $ax^2+\epsilon x^4+ay^2+\epsilon y^4+bz^2=0$ with $a<b$ and $\epsilon > (5a+b)(b-a)/18$ have exactly two umbilics with index $1,$ eight with index $-1/2,$ and eight with index $1/2$ independent of $a$ and $b$?
\end{enumerate}

\section*{Appendix}
\noindent\hypertarget{appendix:1}{[\textbf{1}]} As mentioned, parameterize the surface as $\mathcal{S} = (u, v, (1-au^{2k}-bv^{2k})^{\frac{1}{2k}}$. Then, 
\begin{align*}
    &\Vec{\mathcal{S}}_{u} = (1, 0, -ac^{-\frac{1}{2k}}u^{2k-1}(-au^{2k}-bv^{2k}+1)^{\frac{1}{2k}-1}) \\
    &\Vec{\mathcal{S}}_v = (0, 1, -bc^{-\frac{1}{2k}}u^{2k-1}(-au^{2k}-bv^{2k}+1)^{\frac{1}{2k}-1}) \\
    &\Vec{\mathcal{S}}_{uu} = \left(0, 0, a(2k-1)c^{-\frac{1}{2k}}u^{2k-2}(bv^{2k}-1)(-au^{2k}-bv^{2k}+1)^{-\frac{1}{2k}-2}\right) \\
    &\Vec{\mathcal{S}}_{uv} = \left(0, 0, 2abk(\tfrac{1}{2k}-1)c^{-\frac{1}{2k}}x^{2k-1}y^{2k-1}(-ax^{2k}-by^{2k}+1)^{\frac{1}{2k}-2}\right) \\
    &\Vec{\mathcal{S}}_{vv} = \left(0, 0, b(2k-1)c^{-\frac{1}{2k}}x^{2k-2}(ax^{2k}-1)(-ax^{2k}-by^{2k}+1)^{\frac{1}{2k}-2}\right) \\
    &E = \Vec{\mathcal{S}}_{u}\cdot\Vec{\mathcal{S}}_{u}=1+a^2c^{-\frac{1}{k}}x^{4k-2}(-au^{2k}-bv^{2k}+1)^{\frac{1}{k}-2} \\
    &F = \Vec{\mathcal{S}}_{u}\cdot\Vec{\mathcal{S}}_v = abc^{-\frac{1}{k}}x^{2k-1}y^{2k-1}(-ax^{2k}-by^{2k}+1)^{\frac{1}{k}-2} \\
    & G = \Vec{\mathcal{S}}_v\cdot\Vec{\mathcal{S}}_v = 1+b^2c^{-\frac{1}{k}}v^{4k-2}(-au^{2k}-bv^{2k}+1)^{\frac{1}{k}-2}.
\end{align*}
\begin{align*}
    \Vec{N}\mid\Vec{\mathcal{S}}_u\times\Vec{\mathcal{S}}_v \mid & = \Vec{\mathcal{S}}_u\times\Vec{\mathcal{S}}_v= (ac^{-\frac{1}{2k}}x^{2k-1}(-ax^{2k}-by^{2k}+1)^{\frac{1}{2k}-1}, bc^{-\frac{1}{2k}}y^{2k-1}(-ax^{2k}-by^{2k}+1)^{\frac{1}{2k}-1}, 1)
\end{align*}
$$e = -\Vec{N}\cdot\Vec{\mathcal{S}}_{uu}, \hspace{0.6cm}
f = -\Vec{N}\cdot\Vec{\mathcal{S}}_{uv}, \hspace{0.6cm}
g = -\Vec{N}\cdot\Vec{\mathcal{S}}_{vv}.
$$
\\~\\
\noindent\hypertarget{appendix:2}{[\textbf{2}]}
\onehalfspacing In the following example, set $a=b=c=1, k=2$ and initial conditions $(v,u)=(0,0.7)$ and $v\in[-0.699999,0.699999]$ to avoid singularities at $v=\{-0.7,0.7\}$. We allow for 20 digits of precision and set ``SolveDelayed" to ``True" to avoid singularities.
\singlespacing
\begin{center}
\begin{minipage}{0.46\linewidth}
\begin{lstlisting}[language=Mathematica,columns=fullflexible]
(* Fix a surface*)
a=1; b=1; c=1; k=2
 
(* Set parameterization *)
z = (1 - a (u[v])^(2k) - b v^(2k))^(1/(2k))

(* Calculate derivatives*)
Su = {1, 0, D[z, u[v]]}
Sv = {0, 1, D[z, v]}
Suu = {0, 0, D[z, {u[v], 2}]}
Svv = {0, 0, D[z, {u[v], 2}]}
Suv = {0, 0, D[z, u[v], v]}

(* Calculate fundamental coefficients*)
Es = Dot[Su, Su]
Gs = Dot[Sv, Sv]
Fs = Dot[Su, Sv]
es = -Cross[Su, Sv].Suu
gs = -Cross[Su, Sv].Svv
fs = -Cross[Su, Sv].Suv
\end{lstlisting}
\end{minipage}
\qquad
\begin{minipage}{0.48\linewidth}
\begin{lstlisting}[language=Mathematica,columns=fullflexible,firstnumber=21]
X = fs*Es - es*Fs
Y = gs*Es - es*Gs
Z = gs*Fs - fs*Gs

(* Find lines of curvature with machine precision *)
lowsol = NDSolve[{X(u'[v])^2 + Y(u'[v]) + Z == 0, u[0] == 0.7}, u, {v, -0.699999, 0.699999}, SolveDelayed -> True, InterpolationOrder -> All]

(* Solve with higher precision *)
highsol = NDSolve[{X(u'[v])^2 + Y(u'[v]) + Z == 0, u[0] == 0.7}, u, {v, -0.699999, 0.699999}, SolveDelayed -> True, InterpolationOrder -> All, PrecisionGoal -> 20]

(* Plot results *)
Plot[Evaluate[u[v] /. %], {v, -0.699999, 0.699999}, PlotRange -> {{-1, 1}, {-1, 1}}]
\end{lstlisting} 
\end{minipage}
\end{center}
\noindent\hypertarget{appendix:3}{[\textbf{3}]} \onehalfspacing We overlay several solution curves for the following tested values of $a,b,c,k$. Note that they all have the same general shape, even when $a,b,c,$ or $k$ is relatively large, suggesting a fixed index. Taking advantage of symmetry gives Figure \ref{figure:vecfield}. \singlespacing
\begin{figure}[H]
\begin{minipage}{.32\textwidth}
    \centering
    \includegraphics[scale=0.38]{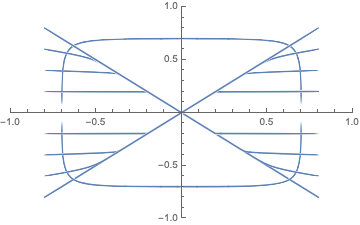}
    \caption*{$a=1,b=1$\\$c=100,k=2$}
\end{minipage}
\begin{minipage}{.32\textwidth}
    \centering
    \includegraphics[scale=0.38]{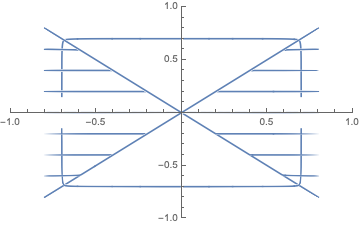}
    \caption*{$a=1,b=1$\\$c=1,k=4$}
\end{minipage}
\begin{minipage}{.32\textwidth}
    \centering
    \includegraphics[scale=0.38]{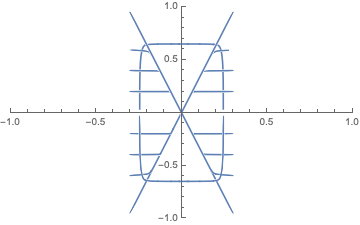}
    \caption*{$a=1,b=10,$\\$c=10,k=2$}
\end{minipage}
\end{figure}
\noindent\hypertarget{appendix:4}{[\textbf{4}]}
\begin{lstlisting}[language=Mathematica, firstnumber = 17,columns=fullflexible]
(*Define residual function*)
residual[v_] = X(u'[v])^2 + Y(u'[v]) + Z

(*Plot log of residuals*)
Plot[Evaluate[RealExponent[{residual[v] /. lowsol, residual[v] /. highsol}]], {v, -0.699999, 0.699999}, PlotStyle -> {GrayLevel[0], RGBColor[1, 0, 0]}, AxesOrigin -> {0, 0}]
\end{lstlisting} 

\noindent\hypertarget{appendix:5}{[\textbf{5}]}
As mentioned, parameterize the surface as $\mathcal{S}=(u,v,b^{-\frac{1}{2}}(1-au^2-\epsilon u^4-av^2-\epsilon v^4)^\frac{1}{2}).$ Then,
\begin{align*}
    \Vec{\mathcal{S}}_{u}&=\left(1,0,b^{-\frac{1}{2}}(1-au^2-\epsilon u^4-av^2-\epsilon v^4)^{-\frac{1}{2}}(-au-2\epsilon u^3)\right) \\
    \Vec{\mathcal{S}}_{v}&=\left(0,1,b^{-\frac{1}{2}}(1-au^2-\epsilon u^4-av^2-\epsilon v^4)^{-\frac{1}{2}}(-av-2\epsilon v^3)\right) \\
    E&=\Vec{\mathcal{S}}_{u}\cdot\Vec{\mathcal{S}}_{u}=1+b^{-\frac{1}{2}}(1-au^2-\epsilon u^4-av^2-\epsilon v^4)^{-1}(-au-2\epsilon u^3)^2 \\
    F&=\Vec{\mathcal{S}}_{u}\cdot\Vec{\mathcal{S}}_{v}=b^{-\frac{1}{2}}(1-au^2-\epsilon u^4-av^2-\epsilon v^4)^{-1}(-au-2\epsilon u^3)(-av-2\epsilon v^3) \\
    G&=\Vec{\mathcal{S}}_{v}\cdot\Vec{\mathcal{S}}_{v}=1+b^{-\frac{1}{2}}(1-au^2-\epsilon u^4-av^2-\epsilon v^4)^{-1}(-av-2\epsilon v^3)^2 \\
    \Vec{\mathcal{S}}_{uu}&=\Big(0,0,-b^{-\frac{1}{2}}(-au-2\epsilon u^3)^2(1-au^2-\epsilon u^4-av^2-\epsilon v^4)^{-\frac{3}{2}}\\&+b^{-\frac{1}{2}}(1-au^2-\epsilon u^4-av^2-\epsilon v^4)^{-\frac{1}{2}}(-a-6\epsilon u^2)\Big) \\
    \Vec{\mathcal{S}}_{vv}&=\Big(0,0,-b^{-\frac{1}{2}}(-av-2\epsilon v^3)^2(1-au^2-\epsilon u^4-av^2-\epsilon v^4)^{-\frac{3}{2}}\\&+b^{-\frac{1}{2}}(1-au^2-\epsilon u^4-av^2-\epsilon v^4)^{-\frac{1}{2}}(-a-6\epsilon v^2)\Big)\\
    \Vec{\mathcal{S}}_{uv}&= \left(0,0,-b^{-\frac{1}{2}}(-au-2\epsilon u^3)(-av-2\epsilon v^3)(1-au^2-\epsilon u^4-av^2-\epsilon v^4)^{-\frac{3}{2}}\right)\\
    \Vec{N}\mid\Vec{\mathcal{S}}_u\times\Vec{\mathcal{S}}_v \mid &= \Vec{\mathcal{S}}_u\times\Vec{\mathcal{S}}_v=\Big(-b^{-\frac{1}{2}}(1-au^2-\epsilon u^4-av^2-\epsilon v^4)^{-\frac{1}{2}}(-au-2\epsilon u^3),\\&-b^{-\frac{1}{2}}(1-au^2-\epsilon u^4-av^2-\epsilon v^4)^{-\frac{1}{2}}(-av-2\epsilon v^3),1 \Big) \\
    e &= -\Vec{N}\cdot\Vec{\mathcal{S}}_{uu}, \hspace{0.6cm}
    f = -\Vec{N}\cdot\Vec{\mathcal{S}}_{uv}, \hspace{0.6cm}
    g = -\Vec{N}\cdot\Vec{\mathcal{S}}_{vv}.
\end{align*}
\\~\\
\noindent\hypertarget{appendix:6}{[\textbf{6}]}
Let $Q=(2\epsilon)^{-1}\left(\sqrt{a^2-4\epsilon(-1+au^2+\epsilon u^4)}-a\right).$ Then,
\begin{align*}
    \Vec{\mathcal{S}}_{u}&=\left(1,0,-(au+2\epsilon u^3)(2\epsilon Q+a)^{-1}Q^{-\frac{1}{2}}\right) \\
    \Vec{\mathcal{S}}_{u}&=\left(0,1,-bv(2\epsilon Q+a)^{-1}Q^{-\frac{1}{2}}\right)=(0,1,0) \\
    E&= \Vec{\mathcal{S}}_{u}\cdot\Vec{\mathcal{S}}_{u}=1+(au+2eu^3)^2(2\epsilon Q+a)^{-2}Q^{-1} \\
    F&=\Vec{\mathcal{S}}_{u}\cdot\Vec{\mathcal{S}}_{v}=bv(au+2eu^3)(2\epsilon Q+a)^{-2}Q^{-1}=0 \\
    G&=\Vec{\mathcal{S}}_{v}\cdot\Vec{\mathcal{S}}_{v}=1+b^2v^2(2\epsilon Q+a)^{-2}Q^{-1}=1 \\
    \Vec{\mathcal{S}}_{uu}&=\left(0,0,-(au+2\epsilon u^3)^2(2\epsilon Q+a)^{-2}Q^{-\frac{3}{2}}(1+4\epsilon Q(2\epsilon Q+a)^{-1})-(a+6\epsilon u^2)(2\epsilon Q+a)^{-1}Q^{-\frac{1}{2}}\right) \\
    \Vec{\mathcal{S}}_{vv}&=\left(0,0,-b^2v^2(2\epsilon Q+a)^{-2}Q^{-\frac{3}{2}}-4b^2ev^2(2\epsilon Q+a)^{-3}Q^{-\frac{1}{2}}-b(2\epsilon Q+a)^{-1}Q^{-\frac{1}{2}}\right)\\&=\left(0,0,-b(2\epsilon Q+a)^{-1}Q^{-\frac{1}{2}}\right) \\
    \Vec{\mathcal{S}}_{uv} &= (0,0,0) \\
    \Vec{N}&\mid\Vec{\mathcal{S}}_u\times\Vec{\mathcal{S}}_v \mid = \Vec{\mathcal{S}}_u\times\Vec{\mathcal{S}}_v=\left((au+2\epsilon u^3)(2\epsilon Q+a)^{-1}Q^{-\frac{1}{2}},0,1\right) \\
    e &= -\Vec{N}\cdot\Vec{\mathcal{S}}_{uu}, \hspace{0.6cm}
    f = -\Vec{N}\cdot\Vec{\mathcal{S}}_{uv}, \hspace{0.6cm}
    g = -\Vec{N}\cdot\Vec{\mathcal{S}}_{vv}.
\end{align*}

\noindent\hypertarget{appendix:7}{[\textbf{7}]}
\begin{center}
\begin{minipage}{0.47\linewidth}
\begin{lstlisting}[language=Mathematica,columns=fullflexible]
(* Fix a surface*)
a = 516/1000; b = 3/10; e = 1/10

(* Set parameterization *)
z = Sqrt[1 - a (u[v])^2 - e (u[v])^4 - a v^2 - e v^4]/Sqrt[b]

(* Calculate derivatives*)
Su = {1, 0, D[z, u[v]]}
Sv = {0, 1, D[z, v]}
Suu = {0, 0, D[z, {u[v], 2}]}
Svv = {0, 0, D[z, {u[v], 2}]}
Suv = {0, 0, D[z, u[v], v]}

(* Calculate fundamental coefficients*)
Es = Dot[Su, Su]
Gs = Dot[Sv, Sv]
Fs = Dot[Su, Sv]
es = -Cross[Su, Sv].Suu
gs = -Cross[Su, Sv].Svv
fs = -Cross[Su, Sv].Suv
\end{lstlisting} 
\end{minipage}
\qquad
\begin{minipage}{0.47\linewidth}
\begin{lstlisting}[language=Mathematica,columns=fullflexible,firstnumber=21]
X = fs*Es - es*Fs
Y = gs*Es - es*Gs
Z = gs*Fs - fs*Gs

(* Find lines of curvature with machine precision *)
lowsol = NDSolve[{X(u'[v])^2 + Y(u'[v]) + Z == 0, u[0] == 0.7}, u, {v, -0.699999, 0.699999}, SolveDelayed -> True, InterpolationOrder -> All]

(* Solve with higher precision *)
highsol = NDSolve[{X(u'[v])^2 + Y(u'[v]) + Z == 0, u[0] == 0.7}, u, {v, -0.699999, 0.699999}, SolveDelayed -> True, InterpolationOrder -> All, PrecisionGoal -> 20]

(* Plot results *)
Plot[Evaluate[u[v] /. %], {v, -0.699999, 0.699999}, PlotRange -> {{-1, 1}, {-1, 1}}]
\end{lstlisting} 
\end{minipage}
\end{center}

\onehalfspacing
\section*{Acknowledgments}
\noindent I would like to express my sincere gratitude to my mentor, Professor Damin Wu at the University of Connecticut, for introducing me to the research topic and for his valuable guidance and helpful advice. Many thanks to my parents for their wholehearted support, and thanks to my sister, cousins, and entire family for keeping me motivated through the pandemic. \\

\newpage
\printbibliography

\end{document}